\documentclass[12pt]{article}
\usepackage{fullpage}
\usepackage{amsmath}
\usepackage{amssymb}
\usepackage{amsthm}

\newtheorem{theorem}{Theorem}

\newtheorem{lemma}{Lemma}

\def\ord{ord}
\def\eps{\varepsilon}

\def\Z{\mathbb{Z}}

\title{The least primitive root modulo $p^{2}$}
\author{Bryce Kerr\footnote{Supported by Australian Research Council Discovery Project DP160100932.}\\ School of Science\\ The University of New South Wales Canberra, Australia \\
  b.kerr@adfa.edu.au
 \and  
  Kevin J. McGown\\ Department of Mathematics and Statistics\\ California State University, Chico\\
The University of New South Wales, Canberra\\
kmcgown@csuchico.edu \and Tim Trudgian\footnote{Supported by Australian Research Council Future Fellowship FT160100094.}\\ School of Science\\ The University of New South Wales Canberra, Australia \\
  t.trudgian@adfa.edu.au}

\begin{document}

\maketitle
\begin{abstract}
\noindent
We provide an explicit estimate on the least primitive root mod $p^{2}$. We show, in particular, that every prime $p$ has a primitive root mod $p^{2}$ that is less than $p^{0.99}$.
\end{abstract}

\section{Introduction}

Let $p$ be an odd prime and let $g(p)$ and $h(p)$ denote the least primitive root modulo $p$ and modulo $p^{2}$, respectively. Burgess \cite{Burgess62} showed that $g(p)\ll p^{1/4 + \epsilon}$ and pointed out in \cite{Burgess71} that the same methods allow one to prove $h(p) \ll p^{1/2 + \epsilon}$. Cohen, Odoni and Stothers \cite{COS} improved this to $h(p) \ll p^{1/4 + \epsilon}$.

While explicit upper bounds on $g(p)$ have been given \cite{Cohen, CohenT, Hunter, MT}, we are unaware of any such bounds for $h(p)$. Trivially, we have $h(p)< p^{2}$, and, if we use an appropriate version of the P\'olya--Vinogradov inequality \cite{Sound}, we can make the estimate $h(p) < p^{1+ \epsilon}$ explicit. This however, is still unable to prove that $h(p) < p$ for all primes $p$.

Even this estimate is probably far from the truth.
If $j$ is a primitive root modulo $p$, then exactly one of the numbers $j+kp$ for $k=0,1,\dots,p-1$ fails to be a primitive root modulo~$p^2$.  As a consequence, one might expect that, heuristically, a primitive root modulo $p$ has a $(p-1)/p$ chance to be a primitive root modulo $p^2$; in particular, when $p$ is large, this is very likely to happen.
Accordingly, Paszkiewicz, see \cite{Pass} and \cite{Liver} proved that $g(p) = h(p)$ for all $p\leq 10^{12}$ except for 
\begin{equation}\label{sink}
p=40,487 \quad \textrm{and} \quad p=6,692,367,337.
\end{equation}

Since we do have good explicit bounds on $g(p)$ it seems reasonable to suggest that these hold for $h(p)$ also.
Indeed, Grosswald \cite{Grosswald81} conjectured that $g(p) < \sqrt{p} -2$ for all $p> 409$ and this is almost certainly true for $h(p)$ as well.
While we are unable to prove this, we can show that all primes $p$ have a primitive root modulo $p^{2}$ less than $p$, which is, in some sense,
the analogue of Grosswald's conjecture for primitive roots modulo $p^2$. In fact, we show slightly more than this in our main result.

\begin{theorem}\label{maintheorem}
We have $h(p) < p^{0.99}$ for all primes $p$.
\end{theorem}
We remark that, just like with $g(p)$, we cannot expect to prove  $h(p) < p^{\alpha}$ for all primes $p$ and for any $\alpha < \log 2/\log 3 \approx 0.6309$
for the simple reason that $2$ is a primitive root mod $9$.
Almost certainly the statement $h(p)\leq p^{\log 2/\log 3}$ is true for all primes, as is the analogue of Grosswald's conjecture: that $h(p) < \sqrt{p} - 2$ for all $p> 409$.


We follow the basic strategy of Cohen, Odoni, and Stothers~\cite{COS}.
Define $P_z=\prod_{p\leq z}p$.
Let $T_1(X)$ denote the positive integers $n\leq X$ satisfying $(n,P_z)=1$ that are primitive roots modulo $p$.
Let $T_2(X)$ denote the positive integers $n\leq X$ satisfying $(n,P_z)=1$ that are $p$-th powers modulo $p^2$.
We will write $N_1(X)=\#T_1(X)$ and $N_2(X)=\#T_2(X)$.
Any element belonging to $T_1\setminus T_2$ is a primitive root modulo $p^2$.
Consequently, if we can choose $z$ so that $N_1(X)-N_2(X)>0$, there exists a primitive root modulo $p^2$ less than $X$.
We seek a lower bound for $N_1(X)$ and an upper bound for $N_2(X)$; the following result supplies the latter.

\begin{theorem}\label{upperboundtheorem}
If $0<\eps<1$ and $z\geq 2^{1/\eps}$, then $N_2(X)\leq p^{1/2}X^{\eps}$.
\end{theorem}

The next theorem immediately gives a lower bound on $N_1$ upon setting $u=P_z$.

\begin{theorem}\label{lowerboundtheorem}
Let $N(X)$ denote the number of primitive roots modulo $p$ less than $X$ that are coprime to $u$.
Let $e$ be an even divisor of $p-1$ and
let $p_1,\dots,p_s$ denote the primes dividing $p-1$ that do not divide $e$.
Set $\delta=1-\sum_{i=1}^sp_i^{-1}$.  Assume $\delta>0$.  Then
\[
  N(X)\geq \delta\theta(e)\left\{
  X
  \sum_{d|u}\frac{\mu(d)}{d}-2^{\omega(u)}\left[
  A(p)
  \left(2+\frac{s-1}{\delta}\right)2^{\omega(p-1)-s}p^{1/2}\log p + 1
  \right]
  \right\}
\]
In the above, $A(p)$ is the constant coming from the P\'olya--Vinogradov inequality (see Lemma~\ref{Lap}).
\end{theorem}

Using Theorems \ref{upperboundtheorem} and \ref{lowerboundtheorem},
we will ultimately show that when $z=16$, $X=p^{0.99}$, one has $N_1(X)-N_2(X)>0$ when $p\geq 10^{12}$.
In light of the computations of Paszkiewicz~\cite{Pass}, and a trivial calculation on the two exceptional primes in (\ref{sink}), this is enough to prove Theorem~\ref{maintheorem}.
In Section~\ref{prelimsection} we quote some preliminary results necessary for our proofs
and in Section~\ref{proofsection} we give proofs of Theorems~\ref{maintheorem},~\ref{upperboundtheorem}, and~\ref{lowerboundtheorem}.

\section{Preliminaries}\label{prelimsection}

We let $f(n)$ denote the primitive root indicator function.
%
We make use of the sieve
employed by Cohen, Olivera e Silva, and Trudgian~\cite{Cohen}
in the form of the following result.

\begin{lemma}\label{lemmasieve}
Let $e$ be an even divisor of $p-1$ and
let $p_1,\dots,p_s$ denote the primes dividing $p-1$ that do not divide $e$.
Set $\delta=1-\sum_{i=1}^sp_i^{-1}$.  Assume $\delta>0$.
Then we have
\begin{eqnarray*}
\frac{f(n)}{\delta\theta(e)}
\geq
1
+\frac{1}{\delta}\sum_{i=1}^s
\theta(p_i)
\sum_{\substack{d|e}}
\frac{\mu(p_i d)}{\phi(p_i d)}
\sum_{\substack{\chi\\ \ord(\chi) = p_{i}d}}
\chi(n)
+
\sum_{\substack{d|e\\d>1}}
\frac{\mu(d)}{\phi(d)}
\sum_{\substack{\chi\\ \ord(\chi) = d}}
\chi(n)
\,.
\end{eqnarray*}
\end{lemma}

The following explicit Polya-Vinogradov is due to Frolenkov and Soundararajan~\cite{Sound}.
\begin{lemma}
\label{lem:pv}
Let $q$ be an integer and $\chi$ a primitive character mod $q$. If $q\ge 1200$ then for any $M$ we have 
\begin{align*}
\left|\sum_{n\le M}\chi(n)\right|\le q^{1/2}\left(\frac{2}{\pi^2}\log{q}+1\right).
\end{align*}
\end{lemma}
We note that a slightly sharper version of Lemma \ref{lem:pv} has been given by Lapkova \cite{Lap}. We also note that, as we only need to consider sums of the form $n\leq M$ (and not $H< n \leq M+H$) we can, via an observation of Pomerance~\cite{Pomerance}, halve the bounds for even characters. Taking, the worst remaining case (when $\chi$ is odd) from \cite[Lem.\ 3A]{Lap} gives the following.
\begin{lemma}\label{Lap}
For $q\geq q_{0}>1$ we have
\begin{equation}
\left|\sum_{n\le M}\chi(n)\right|\le A(q_{0}) \sqrt{q} \log q,
\end{equation}
where $A(q_{0}) = \frac{1}{2\pi} + 0.8204/q_{0}^{1/2} + 1.0285/(q_{0}^{1/2} \log q_{0}).$
\end{lemma}

 The following is due to Rosser and Schoenfeld~\cite{RS}.
\begin{lemma}
\label{lem:eulerphi}
For $n\ge 2$ we have 
\begin{align*}
\frac{n}{\phi(n)}\le e^{\gamma}\log\log{n}\left(1+\frac{2.51}{\log\log{n}} \right).
\end{align*}
In particular, if $n\ge 10^{12}$ then 
\begin{align*}
\frac{n}{\phi(n)}\le 1.8e^{\gamma}\log\log{n}.
\end{align*}
\end{lemma}
The following is due to Robin~\cite{Robin83}.
\begin{lemma}
\label{lem:robin}
For any integer $n\ge 3$ we have 
$$\omega(n)\le 1.3841\frac{\log{n}}{\log\log{n}}.$$
\end{lemma}

\section{Proofs of the theorems}\label{proofsection}

We require two lemmas for the proof of Theorem~\ref{upperboundtheorem}.

\begin{lemma}\label{L:1}
Let $T$ be a subset of $[1,p)$ consisting of $p$th powers modulo $p^2$.
Then
$$\frac{|T|^4}{p}\le \#\{t_1,t_2,t_3,t_4\in T \ : \ t_1t_2=t_3t_4\}.$$
\end{lemma}

\begin{proof}
For $1\le \lambda\le p^2$ let $r(\lambda)$ count the number of solutions to the equation
$$xy\equiv \lambda \pmod{p^2}, \quad x,y\in T.$$
We have 
$$|T|^2=\sum_{1\le \lambda \le p^2}r(\lambda)=\sum_{1\le x \le p}\sum_{1\le y \le p}r(x+yp).$$
By the Cauchy--Schwarz inequality 
\begin{align*}
|T|^4\le p\sum_{1\le x \le p}\left(\sum_{1\le y \le p}r(x+yp) \right)^2=p\sum_{\substack{1\le x \le p \\ 1\le y_1,y_2 \le p}}r(x+y_1p)r(x+y_2p).
\end{align*}
For fixed $1\le x,y_1,y_2\le p$ the term
$$r(x+y_1p)r(x+y_2p),$$
counts the number of solutions to the system of congruences
$$t_1t_2\equiv x+y_1p\pmod{p^2}, \quad t_3t_4\equiv x+y_2p\pmod{p^2}, \quad t_1,t_2,t_3,t_4 \in T.$$
Note that a solution to this system also satisfies
$$t_1t_2-t_3t_4\equiv 0 \pmod{p}.$$
Averaging over $x,y_1,y_2$, we see that 
$$\sum_{\substack{1\le x \le p \\ 1\le y_1,y_2 \le p}}r(x+y_1p)r(x+y_2p)=\#\{ t_1,t_2,t_3,t_4\in T_1 \ : \  t_1t_2\equiv t_3t_4 \pmod{p}\},$$
and hence 
$$\frac{|T|^4}{p}\le \#\{ t_1,t_2,t_3,t_4\in T_1 \ : \  t_1t_2\equiv t_3t_4 \pmod{p}\}.$$ 
As observed in Cohen, Odoni, and Stothers, since $T$ is set of $p$-th powers,
we have that 
$$t_1t_2\equiv t_3t_4 \pmod{p},$$ implies $t_1t_2=t_3t_4$
and hence 
$$\frac{|T|^4}{p}\le \#\{t_1,t_2,t_3,t_4\in T_1 \ : \ t_1t_2=t_3t_4\}.$$
\end{proof}
Cohen, Odoni, and Stothers now proceed to obtain a bound on $\tau_{k}(n)$, the number of ways of writing $n$ as a product of $k$ integers. Proceeding in this way, even with the sharp bounds for $\tau_{k}(n)$ given by Duras \cite{Duras} leads to $h(p)< p$ only for $p> \exp(\exp(58))$.
To reduce this bound we note that $\tau_{k}(n)$ is large only when $n$ is highly composite. If we remove some small prime factors, the bounds improve markedly. We do this below for $\tau(n) = \tau_{2}(n)$.

\begin{lemma}\label{L:2}
Suppose $n\in\Z^+$ is not divisible by any prime $p<z$.
If $0<\eps<1$ and $z\geq 2^{1/\eps}$, then $\tau(n)\leq n^\eps$.
\end{lemma}

\begin{proof}
We follow ideas from the classic argument. Suppose $n$ has prime factorization
$$n=p_1^{\alpha_1}\dots p_k^{\alpha_k},$$
with each $p_i\ge z$. Let $0<\eps<1$ be some parameter
and consider 
\begin{align*}
\frac{\tau(n)}{n^{\eps}}=\prod_{i=1}^{k}\frac{\alpha_i+1}{p_i^{\eps \alpha_i}}.
\end{align*}
If $p_{i}\geq z\ge 2^{1/\eps}$ then 
\begin{align*}
\frac{\alpha_{i}+1}{p_{i}^{\eps \alpha_{i}}}\le \frac{\alpha_{i}+1}{2^{\alpha_{i}}}\le 1,
\end{align*}
whence the lemma follows.
\end{proof}

\begin{proof}[Proof of Theorem~\ref{upperboundtheorem}]
For ease of notation, write $T=T_2(X)$.
By Lemma~\ref{L:1}, 
\begin{align*}
\frac{|T|^4}{p}\le \#\{t_1,\dots,t_4\in T \ : t_1t_2=t_3t_4\}\le |T|^2\max_{n\in TT}\tau(n),
\end{align*}
where $TT$ denotes the product set of $T$. If $t_1,t_2\in T$ then $(t_1t_2,P_z)=1$, and hence by Lemma~\ref{L:2}
we have $\max_{n\in T^2}\tau(n)\leq n^\eps$ for every $n\in TT$.
Since $T\subseteq[1,X]$, it follows that
\begin{align*}
|T|^2\le pn^{\eps}\le p X^{2\eps}.
\end{align*}
\end{proof}

\begin{proof}[Proof of Theorem~\ref{lowerboundtheorem}]
We sum the inequality from Lemma~\ref{lemmasieve} over $n\leq x$ such that $(n, u)=1$
to obtain
\begin{eqnarray*}
\frac{1}{\delta\theta(e)}\sum_n f(n)
\geq
\sum_n 1
+\frac{1}{\delta}\sum_{i=1}^s
\theta(p_i)
\sum_{\substack{d|e}}
\frac{\mu(p_i d)}{\phi(p_i d)}
\sum_{\substack{\chi\\ \ord(\chi) = p_{i}d}}
\sum_n
\chi(n)
+
\sum_{\substack{d|e\\d>1}}
\frac{\mu(d)}{\phi(d)}
\sum_{\substack{\chi\\ \ord(\chi) = d}}
\sum_n
\chi(n)
\,.
\end{eqnarray*}
The two cumbersome terms on the right taken together are bounded above in absolute value by
\[
  \left(1+\frac{1}{\delta}\sum_{i=1}^s \theta(p_i)\right)\sum_{d|e}|\mu(d)|2^{\omega(u)}A(p)p^{1/2}\log p
\]
The reason for the $2^{\omega(u)}$ is that we have written
\[
  \sum_{\substack{n\leq x\\(n,u)=1}}\chi(n)=\sum_{d|u}\mu(d)\chi(d)\sum_{n\leq x/d}\chi(n)
\]
before applying P\'olya--Vinogradov (see Lemma~\ref{Lap}).  
Finally, we observe that
\[
  \sum_{\substack{n\leq x\\(n,u)=1}}1\geq x\sum_{d|u}\frac{\mu(d)}{d}-2^{\omega(u)}
  \,.
\]
Putting all this together gives the result.
\end{proof}

\begin{proof}[Proof of Theorem~\ref{maintheorem}]
Assume $p\geq 10^{12}$ so that $A(p)\leq 0.16$.
Set $z=16$, $u=P_z=2\cdot 3\cdot 5\cdot 7\cdot 11\cdot 13$, and $X=p^{\alpha}$
where $\alpha\leq 1$.
Theorem~\ref{lowerboundtheorem} gives
\[
  N_1(X)\geq \delta\frac{\phi(e)}{e}\left\{
  0.1918p^\alpha
  -2^6\left[
  0.16
  \left(2+\frac{s-1}{\delta}\right)2^{\omega(p-1)-s}p^{1/2}\log p + 1
  \right]
  \right\}.
\]
In the special case where $s=0$ this becomes
\[
  N_1(X)\geq 
  \frac{\phi(p-1)}{p-1}\left\{
  0.1918p^\alpha
  -64\left[
  0.16
  \cdot
  2^{\omega(p-1)}p^{1/2}\log p + 1
  \right]
  \right\}
  \,.
\]
Set $\eps=1/4$ so that the inequality $2^{1/\eps}\leq z$ is satisfied.
Theorem~\ref{upperboundtheorem} then gives
$N_2(X)\leq p^{1/2+\alpha/4}$.
Recall that we must show that $N_1(X)-N_2(X)>0$.
In light of the above, a sufficient condition is
\begin{align}\label{condition}
\delta\frac{\phi(e)}{e}\left\{
  0.1918p^\alpha
  -2^6\left[
  0.16
  \left(2+\frac{s-1}{\delta}\right)2^{\omega(p-1)-s}p^{1/2}\log p + 1
  \right]
  \right\}
  >
  p^{1/2+\alpha/4}.
\end{align}
We now set $\alpha=0.99$.
Applying Lemmas~\ref{lem:eulerphi} and~\ref{lem:robin}, we find that condition~(\ref{condition}) is satisfied with $s=0$ when $p\geq 10^{15}$.
Hence we may assume $p<10^{15}$ and hence $\omega=\omega(p-1)\leq 13$.  When $\omega\leq 8$, again
we find that (\ref{condition}) is satisfied with $s=0$.  The remaining cases are $9\leq\omega\leq 13$.
For each of these cases we choose $s=7$, except when $\omega=14$ we choose $s=8$.
We let $p_1,\dots,p_s$ denote the largest $s$ primes dividing $p-1$.
We have the lower bounds $p_i\geq q_{\omega-s+i}$ where $q_i$ denotes the $i$-th prime;
i.e., $q_1,q_2,q_3,\dots=2,3,5,\dots$.  Using $e<p/(p_1\dots p_s)$ we can then bound
$\phi(e)/e\geq (1.8B\log\log e)^{-1}$ from below, where $B=e^\gamma$ is a constant (the most recent $e$ is Euler's constant).
Similarly, using the lower bounds for the $p_i$ we can bound $\delta$ from below.
All this allows us to verify that (\ref{condition}) holds when $p\geq 10^{12}$.
\end{proof}

\section{Comments}
Our method is configured to prove that for all $p$ we have $h(p)<p^\alpha$ for some $\alpha<1$:
there was nothing special about the choice of $\alpha=0.99$.
We would like to point out that our proof with essentially no modification
shows that $h(p)<p^{0.87}$ when $p\geq 10^{16}$.
Consequently, knowledge of the exceptions to the equation $g(p)=h(p)$ in the range $10^{12}\leq p\leq 10^{16}$
would instantly improve the exponent $\alpha$ in our result from $0.99$ to $0.87$,
as Grosswald's conjecture for $g(p)$ is already known in this range.
Moreover, careful choices of the parameters and some computation
via a tree algorithm similar to that given in~\cite{MTT} and \cite{JT}
could be used to achieve a result with an even smaller $\alpha$.

In addition, we remark that it may be possible to use the methods here to give an explicit bound on $\hat{h}(p)$,
the least primitive root modulo $2p^{2}$, by following the arguments in Elliot and Murata \cite{EM}.
Likely one could prove something like $\hat{h}(p) < p$ for all $p>3$.  It seems reasonable that even $\hat{h}(p) < p^{1/2}$ for all $p>1171$ may be true.

\section*{Acknowledgements}
We wish to thank Jean-Louis Nicolas for making Duras's  thesis \cite{Duras} known to us. We also with to thank the UNSW Canberra Rector's Visiting Fellowship which enabled the second author to visit the first and third authors in February and March 2019.

\end{document}